\documentclass[11pt]{article}
\usepackage{amsmath,amsfonts,amssymb,amsthm,bm,mathtools}
\usepackage[hidelinks]{hyperref}
\usepackage[dvipsnames]{xcolor}
\usepackage[color=Purple!50!white]{todonotes}
\usepackage[capitalize]{cleveref}
\usepackage{fullpage}

\title{An improved lower bound on multicolor Ramsey numbers}
\author{Yuval Wigderson\thanks{Department of Mathematics, Stanford University, Stanford, CA 94305, USA. Email: \url{yuvalwig@stanford.edu}. Research supported by NSF GRFP Grant DGE-1656518.}}

\newtheorem{thm}{Theorem}
\newtheorem*{thm*}{Theorem}
\newtheorem{lem}[thm]{Lemma}

\theoremstyle{definition}

\newtheorem*{rem}{Remark}
\crefname{equation}{Equation}{Equations}
\crefname{thm}{Theorem}{Theorems}

\newcommand\F{\mathbb{F}}

\newcommand\ab[1]{\lvert #1 \rvert}
\DeclareMathOperator\pr{Pr}

\begin{document}
\maketitle
\begin{abstract}
	A recent breakthrough of Conlon and Ferber yielded an exponential improvement on the lower bounds for multicolor diagonal Ramsey numbers. In this note, we modify their construction and obtain improved bounds for more than three colors. 
\end{abstract}
\section{Introduction}
For positive integers $t$ and $\ell$, let $r(t;\ell)$ denote the $\ell$-color Ramsey number of $K_t$, i.e.\ the least integer $N$ such that every $\ell$-coloring of $E(K_N)$ contains a monochromatic $K_t$. The most well-studied case is that of $\ell=2$, where the bounds
\[
	2^{t/2} \leq r(t;2) \leq 2^{2t}
\]
were proved by Erd\H os \cite{Erdos} and Erd\H os--Szekeres \cite{ErSz} in 1947 and 1935, respectively. Despite decades of effort, only lower-order improvements have been made to these bounds \cite{Spencer,Conlon,Sah}.

For larger values of $\ell$, even less is known. The Erd\H os--Szekeres \cite{ErSz} argument yields that $r(t;\ell) \leq \ell^{\ell t}$. For the lower bound, Erd\H os's random construction \cite{Erdos} shows that $r(t;\ell) \geq \ell^{t/2}$. This was improved substantially by Lefmann \cite{Lefmann}, who used an iterated product coloring to show that $r(t;\ell) \geq 2^{t\ell/4}$. Thus, we see that the dependence on the clique size $t$ is exponential, and the dependence on the number of colors $\ell$ is somewhere between exponential and super-exponential, i.e.\ between $2^{\Omega(\ell)}$ and $2^{O(\ell \log \ell)}$. It is a major open problem to determine the correct $\ell$-dependence. Already for the case $t=3$, Erd\H os offered \$100 for the determination of whether $r(3;\ell)$ is exponential or super-exponential in $\ell$, and this question is closely related to a number of other questions in graph theory, coding theory, and beyond; see e.g.\ \cite{Alon,NeRo} for more.

In a recent breakthrough, Conlon and Ferber \cite{CoFe} improved Lefmann's lower bound on $r(t;\ell)$ for fixed $\ell>2$ and $t \to \infty$. To do so, they introduced a new construction that mixes algebraic and probabilistic approaches, and which does better than the random construction for $\ell=3$ and $\ell=4$. Then, they use Lefmann's iterated product trick to obtain better bounds for all larger values of $\ell$ as well. Their result is that
\[
	r(t;\ell) \geq \left( 2^{\frac{7\ell}{24}+C} \right) ^{t-o(t)},
\]
for some constant $C$ that depends only on the residue of $\ell$ modulo 3. In this note, we use a variant of the Conlon--Ferber construction to improve the lower bounds on $r(t;\ell)$ for fixed $\ell$ and large $t$.
\begin{thm}\label{thm:high-level}
	For any fixed $\ell \geq 2$,
	\[
		r(t;\ell) \geq \left( 2^{\frac{3\ell}{8}- \frac 14} \right) ^{t-o(t)}.
	\]
\end{thm}
Theorem \ref{thm:high-level} gives the best known bound for all $\ell \geq 4$, and for large $\ell$, improves the constant in the exponent by roughly a factor of $9/7$. It is interesting to note that for our bound, we do not use a product coloring at all, and instead obtain the bound in \cref{thm:high-level} directly from the construction. The bound also matches the best known exponential constant for $\ell=2$ (due to Erd\H{o}s \cite{Erdos}) and for $\ell=3$ (due to Conlon and Ferber \cite{CoFe}). This is because our construction specializes for $\ell=2,3$ to these earlier constructions.

At a high level, our construction differs from the Conlon--Ferber construction by replacing their random induced subgraph by a number of independent random blowups. Such an approach to proving lower bounds for multicolor Ramsey problems goes back to work of Alon and R\"odl \cite{AlRo}. Moreover, it was observed in \cite{HeWi}, combining ideas of Alon--R\"odl with those of Mubayi--Verstra\"ete \cite{MuVe}, that for such problems random induced subgraphs and random blowups are closely related, and are both part of a more general framework of random homomorphisms.

\section{Proof of Theorem \ref{thm:high-level}}

We begin by recalling the basics of the Conlon--Ferber construction, in the special case of $q=2$. Let $t$ be even and let $V \subset \F_2^t$ denote the set of vectors of even Hamming weight, so that $\ab V=2^{t-1}$. We define a graph $G_0$ with vertex set $V$ by letting $\{u,v\} \in E(G_0)$ if and only if $u \cdot v=1$, where $u \cdot v = \sum_{i=1}^t u_i v_i$ denotes the scalar product over $\F_2$. 

\begin{lem}[Conlon--Ferber \cite{CoFe}]\label{lem:oddtown}
	$G_0$ has no clique of order $t$.
\end{lem}
\begin{proof}
	This is a simple variant of the Oddtown theorem \cite{Berlekamp}. Since $V$ consists of vectors of even Hamming weight, we see that $v \cdot v=0$ for all $v \in V$. Therefore, it is simple to show that every clique in $G_0$ of order $t$ consists of linearly independent vectors, since $t$ is even. Since $\dim V=t-1$, this gives the desired result. 
\end{proof}

\begin{lem}[Conlon--Ferber \cite{CoFe}]\label{lem:ind-number-bound}
	$G_0$ has at most $2^{\frac{5t^2}8+o(t^2)}$ independent sets of order at most $t$.
\end{lem}

In their paper, Conlon and Ferber only state this bound for the number of independent sets of size exactly $t$, but their proof actually yields \cref{lem:ind-number-bound}.

We now fix a non-negative integer $m$, and define an $(m+2)$-coloring $\chi$ of $E(K_N)$ for every $N$. We will eventually take $N=2^{\frac{3mt}{8}+\frac t2-o(t)}$; in particular, one should think of $N$ as much larger than $\ab V = 2^{t-1}$. We pick $m$ uniformly random functions $f_1,\ldots,f_m:[N] \to V$, all independent of one another. For two distinct vertices $x,y \in [N]$, we define their color $\chi(x,y)$ as follows. First, if there is some index $i \in [m]$ such that $\{f_i(x),f_i(y)\} \in E(G_0)$, then we let $\chi(x,y)$ be the minimum such index $i$; note that in particular, $\chi(x,y)=i$ implies that $f_i(x) \neq f_i(y)$. If there is no such $i$, then we pick $\chi(x,y) \in \{m+1,m+2\}$ uniformly at random, with these choices made independently over all pairs $x,y$. 

In other words, the coloring of $K_N$ is obtained by overlaying $m$ random blowups of $G_0$ to $N$ vertices, and then randomly coloring all the remaining pairs with the two unused colors. We now claim that for an appropriate choice of $N$, this coloring will contain no monochromatic cliques of order $t$.

\begin{thm}\label{thm:main}
	For every non-negative integer $m$, if $N=2^{\frac{3mt}{8}+\frac t2-o(t)}$, then the coloring $\chi$ will contain no monochromatic clique of order $t$ with positive probability. In particular, $r(t;m+2) \geq 2^{\frac{3mt}{8}+\frac t2-o(t)}$.
\end{thm}
\begin{rem}
	By letting $m=\ell-2$, one obtains the bound in \cref{thm:high-level}.
\end{rem}
\begin{proof}
	We fix a set $S \subset [N]$ with $\ab S=t$, and will bound the probability that $S$ spans a monochromatic clique under $\chi$. First, we observe that $S$ cannot be a monochromatic clique in any of the first $m$ colors, since blowing up a graph cannot increase its clique number. More formally, if $S$ were a monochromatic clique in color $i \in [m]$, then the set of vertices $f_i(S) \subset V$ would form a clique in $G_0$ of order $t$, which cannot exist by \cref{lem:oddtown}.

	Now we bound the probability that $S$ is monochromatic in one of the last two colors. To do so, we will first compute the probability that no pair in $S$ receives one of the first $m$ colors, i.e.\ the probability that all the functions $f_1,\ldots,f_m$ map $S$ into an independent set of $G_0$. If $T$ is some independent set of $G_0$ with $\ab T\leq t$, then the probability that $f_i(S) \subseteq T$ is precisely $(\ab T/\ab V)^t$, since each vertex of $S$ has a $\ab T/\ab V$ chance of being mapped into $T$ by $f_i$. Therefore,
	\[
		\pr(f_i(S) = T) \leq \left( \frac{\ab T}{\ab V} \right) ^t\leq \left( \frac{t}{2^{t-1}} \right) ^t=2^{-t^2+o(t^2)}.
	\]
	By \cref{lem:ind-number-bound}, the number of choices for such a $T$ is at most $2^{5t^2/8+o(t^2)}$. Therefore, by the union bound, the probability that $f_i(S)$ is an independent set in $G_0$ is at most $2^{-3t^2/8+o(t^2)}$. Since these events are independent over all $i \in [m]$, we conclude that
	\[
		\pr(f_i(S)\text{ is independent in }G_0\text{ for all }i \in [m]) \leq 2^{-\frac{3mt^2}{8}+o(t^2)}.
	\]
	Now, for $S$ to be monochromatic in one of the last two colors, we must first have that $f_i(S)$ is independent in $G_0$ for all $i$, and then that all the pairs in $S$ receive the same color under the random assignment of the colors $m+1$ and $m+2$. In other words,
	\begin{align*}
		\pr(S\text{ is monochromatic}) &= 2^{1-\binom t2}\pr(f_i(S)\text{ is independent in }G_0\text{ for all }i \in [m])\\
		&\leq 2^{-\frac{t^2}{2}- \frac{3mt^2}{8}+o(t^2)}.
	\end{align*}
	Finally, we can apply the union bound over all choices of $S$, and conclude that
	\begin{align*}
		\pr(K_N\text{ has a monochromatic clique of order }t) &\leq \binom Nt 2^{-\frac{t^2}{2}- \frac{3mt^2}{8}+o(t^2)}\\
		&\leq \left( N 2^{-\frac{t}{2}-\frac{3mt}{8}+o(t)} \right) ^t\\
		&=o(1),
	\end{align*}
	by our choice of $N=2^{\frac{3mt}{8}+\frac t2-o(t)}$.
\end{proof}

For $m=1$, our construction is actually identical to the Conlon--Ferber construction, so of course yields their bound of $r(t;3) \geq 2^{7t/8-o(t)}$. However, already for four colors our construction starts doing better than theirs. Specifically, applying \cref{thm:main} to $m=2$, we conclude that
\[
	r(t;4) \geq 2^{\frac{6t}{8}+\frac{t}{2}-o(t)} = 2^{\frac{5t}{4}-o(t)} \approx 2.37^t,
\]
whereas their lower bound is roughly $2.13^t$. Additionally, one can check that our bound is stronger than the Conlon--Ferber bound for all $\ell \geq 4$. 

\paragraph{Acknowledgments.} I would like to thank David Conlon and Asaf Ferber for helpful comments, and Xiaoyu He for introducing me to the method of random homomorphisms. I am also extremely grateful to Jacob Fox for many insights and for carefully reading an earlier draft of this paper.


\begin{thebibliography}{10}
\providecommand{\url}[1]{\texttt{#1}}
\providecommand{\urlprefix}{URL }
\providecommand{\eprint}[2][]{\url{#2}}

\bibitem{Alon}
N.~Alon, Lov\'asz, vectors, graphs and codes, in I.~B\'ar\'any, G.~Katona, and
  A.~Sali (eds.), \emph{Building Bridges {II}}, \emph{Bolyai Soc. Math. Stud.},
  vol.~28, Springer, 2019.

\bibitem{AlRo}
N.~Alon and V.~R\"{o}dl, Sharp bounds for some multicolor {R}amsey numbers,
  \emph{Combinatorica} \textbf{25} (2005), 125--141.

\bibitem{Berlekamp}
E.~R. Berlekamp, On subsets with intersections of even cardinality,
  \emph{Canad. Math. Bull.} \textbf{12} (1969), 471--474.

\bibitem{Conlon}
D.~Conlon, A new upper bound for diagonal {R}amsey numbers, \emph{Ann. of Math.
  (2)} \textbf{170} (2009), 941--960.

\bibitem{CoFe}
D.~Conlon and A.~Ferber, Lower bounds for multicolor {R}amsey numbers, 2020.
  Preprint available at arXiv:2009.10458.

\bibitem{Erdos}
P.~Erd\"{o}s, Some remarks on the theory of graphs, \emph{Bull. Amer. Math.
  Soc.} \textbf{53} (1947), 292--294.

\bibitem{ErSz}
P.~Erd\"{o}s and G.~Szekeres, A combinatorial problem in geometry,
  \emph{Compositio Math.} \textbf{2} (1935), 463--470.

\bibitem{HeWi}
X.~He and Y.~Wigderson, Multicolor {R}amsey numbers via pseudorandom graphs,
  \emph{Electron. J. Combin.} \textbf{27} (2020), Article No. P1.32.

\bibitem{Lefmann}
H.~Lefmann, A note on {R}amsey numbers, \emph{Studia Sci. Math. Hungar.}
  \textbf{22} (1987), 445--446.

\bibitem{MuVe}
D.~Mubayi and J.~Verstra\"ete, A note on pseudorandom {R}amsey graphs, 2019.
  Preprint available at arXiv:1909.01461.

\bibitem{NeRo}
J.~Ne\v{s}et\v{r}il and M.~Rosenfeld, I. {S}chur, {C}. {E}. {S}hannon and
  {R}amsey numbers, a short story, \emph{Discrete Math.} \textbf{229} (2001),
  185--195.

\bibitem{Sah}
A.~Sah, Diagonal {R}amsey via effective quasirandomness, 2020. Preprint
  available at arXiv:\allowbreak 2005.09251.

\bibitem{Spencer}
J.~Spencer, Ramsey's theorem---a new lower bound, \emph{J. Combin. Theory Ser.
  A} \textbf{18} (1975), 108--115.

\end{thebibliography}
\end{document}